\documentclass[12pt,reqno]{amsart}
\usepackage[utf8]{inputenc}
\usepackage[T1]{fontenc}
\usepackage[usenames, dvipsnames]{color}
\usepackage{ulem}

\usepackage{dsfont, amsfonts, amsmath, amssymb,amscd, stmaryrd, latexsym, amsthm}
\usepackage[frenchb,english]{babel}
\usepackage{enumerate}
\usepackage{longtable}
\usepackage{geometry}
\usepackage{float}
\usepackage{tikz}
\usetikzlibrary{shapes,arrows}
\usepackage{float}
\usepackage{tikz}
\usepackage{xypic}
\usetikzlibrary{shapes,arrows}

\newtheorem{theorem}{Theorem}[section]
\newtheorem{lemma}[theorem]{Lemma}
\newtheorem{proposition}[theorem]{Proposition}
\newtheorem{corollary}[theorem]{Corollary}

\theoremstyle{remark}

\newtheorem{example}[theorem]{\bf Example}

\usepackage{hyperref}
\hypersetup{
	colorlinks=true,
	urlcolor=blue,
	citecolor=blue}
\def\int{\mathrm{Int}}
\begin{document}
\selectlanguage{english}
\title[Krull dimension of rings of integer-valued polynomials]{Note on the Krull dimension of rings of integer-valued polynomials}

\author[M.M. Chems-Eddin]{M.M. Chems-Eddin}
\address{Mohamed Mahmoud CHEMS-EDDIN: Department of Mathematics, Faculty of Sciences Dhar El Mahraz, Sidi Mohamed Ben Abdellah University, Fez,  Morocco}
\email{2m.chemseddin@gmail.com}
 
\author[B. Feryouch]{B. Feryouch}
\address{Badr FERYOUCH: Department of Mathematics, Faculty of Sciences Dhar El Mahraz, Sidi Mohamed Ben Abdellah University, Fez,  Morocco}
\email{badr.feryouch@usmba.ac.ma}

\author[A. Tamoussit]{A. Tamoussit}
\address{Ali TAMOUSSIT: Department of Mathematics, The Regional Center for Education and Training Professions Souss-Massa,  Inezgane, Morocco}
\email{a.tamoussit@crmefsm.ac.ma ; tamoussit2009@gmail.com}

\subjclass[2020]{13C15, 13F05, 13F20.}

\keywords{Integer-valued  polynomials, Krull dimension,  Pseudo-valuation domain.}
\date{\today}
	
\maketitle
\begin{abstract} 
Let $D$ be an integral domain with quotient field $K,$ $E$ a subset of $K$ and $X$ an indeterminate over $K$. The set $\mathrm{Int}(E,D):=\{f\in K[X];\; f(E)\subseteq D\}$, of \textit{integer-valued polynomials on $E$ over $D$}, is known to be an integral domain. The purpose of this note is to calculate the Krull dimension of  $\mathrm{Int}(E,D)$ across various classes of integral domains $D$ and specific subsets $E$ of $D$. We further extend our study to the ring $\mathrm{Int}_B(E,D):=\{f\in B[X];\; f(E)\subseteq D\},$ where $B$ is an integral domain containing $D$.
\end{abstract}
	
\section*{Introduction}

The study of the Krull dimension of polynomial rings is a classical topic in commutative ring theory. For the most common classes of integral domains, such as Noetherian domains and Pr\"ufer domains, one has the equality $\dim(D[X])=\dim(D)+1$. However, the Krull dimension of $D[X]$ can be strictly larger than $1 + \dim (D)$ for more exotic rings that are not part of the family of Jaffard domains. Another classical object of study in multiplicative ideal theory is the ring of integer-valued polynomials, which has many interesting properties. A similar study of the Krull dimension, like that conducted for polynomial rings, has also appeared in the literature in the context of rings of integer-valued polynomials. In the following paragraph, we will recall these rings.

\medskip

Let $D$ be an integral domain and $K$ its quotient field. The \textit{ring of integer-valued polynomials} (\textit{of} $D$) is defined as follows: $$\mathrm{Int}(D)= \{f\in K[X];\; f(D)\subseteq D\}.$$ The ring $\mathrm{Int}(\mathbb{Z})$ was first studied by Ostrowski \cite{O19} and P\'{o}lya \cite{P19} in 1919 and, after a century of research, it has become a classical object in commutative ring theory, number theory, and other areas of active research in mathematics (see \cite{C97} and references therein). In \cite{C93}, Cahen considered the ring of \textit{integer-valued polynomials over a subset} $E$ of $K$, defined by: $$\mathrm{Int}(E,D)= \{f\in K[X];\; f(E)\subseteq D\}.$$  This ring generalizes the classical ring of integer-valued polynomials (because $\mathrm{Int}(D)=\mathrm{Int}(D,D)$). It is evident that the inclusions $D[X]\subseteq\mathrm{Int}(D)\subseteq K[X]$  and $D\subseteq\mathrm{Int}(E,D)\subseteq K[X]$ always hold.

\medskip

The rings of integer-valued polynomials are an important source for providing examples of rings that are non-Noetherian with finite Krull dimension, which are generally difficult to find. In this context, the first example that comes to mind is the classical ring $\mathrm{Int}(\mathbb{Z})$, which is known to be a two-dimensional non-Noetherian domain \cite{C97}. Notably, for a finite subset $E$ of the quotient field of an integral domain $D$, the ring $\mathrm{Int}(E,D)$ is non-Noetherian with Krull dimension equal to $1+\dim(D)$ (see \cite[Lemma 4.1]{C93} and \cite[Lemma 4.1\rm(a)]{FIKT97}). Therefore, by starting with an integral domain $D$ of finite Krull dimension and choosing a finite subset $E$ of $K:= \mathrm{qf}(D)$, one can construct a non-Noetherian ring with finite Krull dimension.

\medskip
 	
One of the most challenging open problems related to the ring $\mathrm{Int}(D)$, or even $\mathrm{Int}(E,D)$, is determining its Krull dimension. Although several satisfactory results have been obtained, as noted in the introduction of \cite{FIKT97}, finding a general upper bound for the Krull dimension of $\mathrm{Int}(D)$, in terms of $\dim(D)$, remains unsolved. This unresolved problem has prompted many researchers to delve deeper into the subject for further insights. It is also worth noting that the Krull dimension of rings of the form $\mathrm{Int}(E,D)$ has been studied in the literature much less than the classical ring $\mathrm{Int}(D)$. This is due to the intrinsic difficulties arising from the impact that the subset $E$ can have on the properties of $\mathrm{Int}(E,D)$, especially when viewed in relation to $\mathrm{Int}(D)$ (for instance, the good behavior under localization and the triviality).

\medskip

Before presenting some well-known results, we recall that the \textit{valuative dimension} of an integral domain $D$, denoted by $\dim_v(D)$, is defined as the supremum of the Krull dimensions of all valuation overrings of $D$. It is worth noting that the inequalities $\dim(D) \leqslant \dim_v(D)$ and $\dim_v(T) \leqslant \dim_v(D)$ hold for any integral domain $D$ and any overring $T$ of $D$.

\medskip

In 1997, Cahen and Chabert \cite[Proposition V.1.8]{C97} proved that for any fractional subset $E$ of a Jaffard domain $D$, $\dim(\mathrm{Int}(E,D))=1+\dim(D)$. Additionally, they showed that $\dim_v(\mathrm{Int}(E,D))=1+\dim_v(D)$ for any fractional subset $E$ of $D$; see \cite[Exercise V.5.(ii)]{C97}. Around the same time, Tartarone \cite{T97} pointed out that there are other important results in the case of rings obtained by pullback diagrams and studied the important class of pullback rings given by the pseudo-valuation domains.  Thereafter, in 2004, Fontana and Kabbaj \cite{FK04} established that $\dim(\mathrm{Int}(D))=\dim(D[X])$ for any locally essential domain $D$. More recently, in \cite{COT23}, the authors proved that for certain Jaffard-like domains $D$, the Krull dimension of any ring between $D[X]$ and $\mathrm{Int}(E,D)$ is equal to the Krull dimension of $D[X]$, under the assumption that the subset $E$ is residually cofinite with $D$. Remarkably, this result recovers the previous finding by Fontana and Kabbaj \cite[Theorem 2.1]{FK04}.

\medskip

In the definition of $\mathrm{Int}(E,D)$, if we replace $K[X]$ by $B[X]$, where $B$ is an integral domain containing $D,$ we obtain a new class of rings denoted by $\mathrm{Int}_B(E,D),$ which is called the \textit{ring of $D$-valued $B$-polynomials on} $E$ \textit{over} $D.$ Specifically, $$\mathrm{Int}_B(E,D):=\{f\in B[X];\;f(E)\subseteq D\}.$$
These rings were introduced and studied by Tamoussit in \cite{AT2021}. For simplicity, we write $\mathrm{Int}(E,D)$ and $\mathrm{Int}(D)$ instead of $\mathrm{Int}_K(E,D)$ and $\mathrm{Int}_K(D)$, respectively.
\bigskip	 	

In this paper we aim to dig a little deeper into the study of the Krull dimension of $\mathrm{Int}(E,D)$, and, more generally $\int_B(E,D),$ providing new results and illustrative examples. In particular, we summarize our main results as follows.

We first compare $\dim(\int(E, D))$ with $\dim(D_{\mathfrak{p}}[X])$ for a divided prime ideal $\mathfrak{p}$ of $D$ with infinite residue field (Theorem \ref{Divides2}). Next, we provide sufficient conditions for $\dim(\int(E, D)) = \dim(B[X])$, where $B$ is a specific type of overring of $D$ that shares an ideal $I$ with $D$ such that $D/I$ is finite (Theorem \ref{pairs domain}). We then compare $\dim(\int(E, D))$ with the valuative dimensions $\dim_v(W_{\mathfrak{m}})$, where $D$ has finite character, $\mathfrak{m}$ ranges over the maximal ideals of $D$, and $W_{\mathfrak{m}}$ is an overring of $D_{\mathfrak{m}}$ that shares the ideal $\mathfrak{m}D_{\mathfrak{m}}$ (Theorem \ref{lpsvlf}). Lastly, we examine the Krull dimension of $\int_B(E, D)$ (Theorem \ref{main1}).

\medskip

Throughout this paper $D$ is an integral domain with quotient field $K$ and $E$ is a nonempty subset of $K$.

\section{Main results and examples}\label{sec1}
Recall that a prime ideal $\mathfrak{p}$ of an integral domain $D$ is \textit{divided} if $\mathfrak{p}=\mathfrak{p}D_{\mathfrak{p}}$. It is well-known that all prime ideals of a valuation domain are divided.

\medskip

We begin this section by generalizing \cite[Corollary 2.2]{T97} to the setting of $\mathrm{Int}(E,D)$, under certain conditions on the subsets $E$ of $D$. In \cite[Corollary 2.2]{T97}, Tartarone proved that for any divided prime ideal $\mathfrak{p}$ of an integral domain $D$ with an infinite residue field, the equality $\dim(\mathrm{Int}(D))=\dim(D_{\mathfrak{p}}[X])+\dim(\mathrm{Int}(D/\mathfrak{p}))-1$ holds. Our first main result extends this equality to the case of $\mathrm{Int}(E,D)$ by considering an additional condition that relates the subset $E$ to the prime ideals of $D$. 

\medskip

In \cite{Mu95} Mulay introduced and studied the concept of \textit{residual cofiniteness} in the study of integer-valued polynomials over subsets. A simple characterization of this notion is given in \cite[Lemma 2]{Mu95}, which states that a nonempty subset $E$ of $D$ is residually cofinite with $D$ if and only if the following condition holds: \begin{center}
for each prime ideal $\mathfrak{p}$ of $D$, the condition $|E/\mathfrak{p}| < \infty$ implies $|D/\mathfrak{p}| < \infty$.
\end{center} Clearly, $D$ is residually cofinite with itself.

\begin{lemma}\label{Triv}
Let $D$ be an integral domain, $\mathfrak{p}$  a prime ideal of $D$ with infinite residue field and $E$ a  nonempty subset of $D$. Assume that $E$ satisfies one of the following three conditions:
\begin{enumerate}[\hspace*{0.5cm} $(a)$]
	\item $E$ meets infinitely many cosets of $\mathfrak{p}$.
	\item $E$ is an ideal of $D$ containing $\mathfrak{p}$.
	\item $E$ is residually cofinite with $D$.
\end{enumerate}
Then the following statements hold:
\begin{enumerate}[$(1)$]
\item  $\mathrm{Int}(E,D)$ is contained in $D_{\mathfrak{p}}[X].$ In fact, $\mathrm{Int}(E,D)_\mathfrak{p}=D_{\mathfrak{p}}[X].$
\item If $\mathfrak{p}$ is a divided prime ideal of $D$, then the canonical map $$\begin{array}{ccccc}
\varphi& : & \mathrm{Int}(E,D) & \longrightarrow & \mathrm{Int}(E/\mathfrak{p},D/\mathfrak{p}) \\
 & & f & \longmapsto & \bar{f}:=f+\mathfrak{p}D_{\mathfrak{p}}[X]
\end{array}$$ is a surjective homomorphism with kernel equal to $\mathfrak{p}[X],$ and hence, $$\mathrm{Int}(E,D)/\mathfrak{p}[X]\cong\mathrm{Int}(E/\mathfrak{p},D/\mathfrak{p}).$$
\end{enumerate}
\end{lemma}

\begin{proof}
\rm(1) If (a) or (b) holds, the desired conclusion follows from \cite[Remarks I.3.5(ii)]{C97} and the comment preceding \cite[Lemma 3.1]{FIKT00}, respectively.\\
When (c) holds, \cite[Lemma 4(ii)]{Mu95} implies that $\mathrm{Int}(E, D_\mathfrak{p}) = D_\mathfrak{p}[X]$. Consequently, $\mathrm{Int}(E, D)$ is contained in $D_\mathfrak{p}[X]$ since $\mathrm{Int}(E, D_\mathfrak{p})$ is an overring of $\mathrm{Int}(E, D)$. 

\smallskip

\noindent \rm(2) Note first that, by the previous statement, we have $\mathrm{Int}(E, D) \subseteq D_{\mathfrak{p}}[X]$, which guarantees that the map $\varphi$ is well-defined. Now assume that $\mathfrak{p}$ is a divided prime ideal of $D$ and let $g \in \mathrm{Int}(E/\mathfrak{p}, D/\mathfrak{p})$. Then $g$ is the class of a polynomial $f$ of $D_{\mathfrak{p}}[X].$ In fact,  $f\in \mathrm{Int}(E,D)$ because $f(E)\subseteq D+\mathfrak{p}D_{\mathfrak{p}}$ and $\mathfrak{p}$ is divided. Moreover, the kernel of $\varphi$ is $\mathfrak{p}D_{\mathfrak{p}}[X] \cap \mathrm{Int}(E, D)$, which equals $\mathfrak{p}[X]$. This last equality follows from the fact that $\mathfrak{p}$ being divided implies $\mathfrak{p}D_{\mathfrak{p}}[X] \cap \mathrm{Int}(E, D) = \mathfrak{p}[X] \cap \mathrm{Int}(E, D) = \mathfrak{p}[X]$. Therefore, $\mathrm{Int}(E,D)/\mathfrak{p}[X]\cong\mathrm{Int}(E/\mathfrak{p},D/\mathfrak{p}).$
\end{proof}

Before presenting our first main result, we recall a well-known fact that will be used in the proof. Specifically, if $\mathfrak{p}$ is a divided prime ideal of an integral domain $D$, then $\dim(D_{\mathfrak{p}}[X]) = \mathrm{ht}_{D_\mathfrak{p}[X]}(\mathfrak{p}[X]) + 1.$ Indeed, let $\mathfrak{M}$ be a maximal ideal of $D_\mathfrak{p}[X]$ such that $\dim(D_{\mathfrak{p}}[X]) = \mathrm{ht}_{D_\mathfrak{p}[X]}(\mathfrak{M})$. Then, $\mathfrak{M} \cap D_\mathfrak{p}$ is a maximal ideal of $D_\mathfrak{p}$, and hence  $\mathfrak{M} \cap D_\mathfrak{p} = \mathfrak{p}D_\mathfrak{p} = \mathfrak{p}$ because $\mathfrak{p}$ is a divided prime ideal of $D$. Moreover, as $\mathfrak{p}[X]$ is not a maximal ideal of $D_\mathfrak{p}[X]$, it follows that $\mathfrak{M} \neq \mathfrak{p}[X].$ Therefore,  $\mathrm{ht}_{D_\mathfrak{p}[X]}(\mathfrak{M}) = \mathrm{ht}_{D_\mathfrak{p}[X]}(\mathfrak{p}[X]) + 1,$ which proves the desired equality.

\begin{theorem}\label{Divides2}
Under the same conditions of Lemma \ref{Triv}, we have  $$\dim(\mathrm{Int}(E,D))=\dim(D_{\mathfrak{p}}[X])+\dim(\mathrm{Int}(E/\mathfrak{p},D/\mathfrak{p}))-1.$$
Moreover, if $\mathfrak{p}$ is maximal, then $\dim(\mathrm{Int}(E,D))=\dim(D_{\mathfrak{p}}[X])$, and hence, $$\dim(\mathrm{Int}(E,D))\leqslant \dim(D[X]).$$
\end{theorem}

\begin{proof}
We will proceed similarly to the proof of \cite[Corollary 2.2]{T97}. Since $\mathfrak{p}$ is divided, the domains $D$ and $D_{\mathfrak{p}}$ share the ideal $\mathfrak{p}$, and so the following diagram is a pullback:
$$\xymatrix{
    D \ar[r] \ar[d]  & D/\mathfrak{p} \ar[d] \\
    D_{\mathfrak{p}} \ar[r] & D_{\mathfrak{p}}/\mathfrak{p}D_{\mathfrak{p}}.
  }$$  
Hence, the following diagram is also a pullback:
$$\xymatrix{
    D[X] \ar[r] \ar[d]  & D[X]/\mathfrak{p}[X] \ar[d] \\
    D_{\mathfrak{p}}[X] \ar[r] & D_{\mathfrak{p}}[X]/\mathfrak{p}D_{\mathfrak{p}}[X].
  }$$
By Lemma \ref{Triv}(1), we have the inclusion $\mathrm{Int}(E,D)\subseteq D_{\mathfrak{p}}[X],$ and thus we can write the following:
 $$\xymatrix{
    D[X] \ar[r] \ar[d]  & D[X]/\mathfrak{p}[X] \ar[d] \\
     \mathrm{Int}(E,D)\ar[r] \ar[d]  & \mathrm{Int}(E,D)/\left(\mathfrak{p}D_{\mathfrak{p}}[X]\cap \mathrm{Int}(E,D)\right) \ar[d]\\
    D_{\mathfrak{p}}[X] \ar[r] & D_{\mathfrak{p}}[X]/\mathfrak{p}D_{\mathfrak{p}}[X].
  }$$ 
Consider the set $\mathcal{P}$ of all prime ideals of $D_{\mathfrak{p}}[X]$ containing $\mathfrak{p}D_{\mathfrak{p}}[X].$ So, by \cite[Théorème 1]{C88}, we have
 $$\dim(\mathrm{Int}(E,D))\leqslant \sup_{{\mathfrak{Q}\in \mathcal{P}}}\lbrace\dim(D_{\mathfrak{p}}[X]),\mathrm{ht}_{D_{\mathfrak{p}}[X]}(\mathfrak{Q})+\dim(\mathrm{Int}(E,D)/\left(\mathfrak{Q}\cap \mathrm{Int}(E,D)\right) \rbrace.$$
Now, let $\mathfrak{Q}$ be an element of $\mathcal{P}$. We will then discuss the following two possible cases:

\medskip

\textbf{Case 1:} $\mathfrak{Q}=\mathfrak{p}[X].$ In this case, as noticed just after the proof of Lemma \ref{Triv}, we have $\dim(D_{\mathfrak{p}}[X])=\mathrm{ht}_{D_{\mathfrak{p}}[X]}(\mathfrak{Q})+1,$ and then $\mathrm{ht}_{D_{\mathfrak{p}}[X]}(\mathfrak{Q})=\dim(D_{\mathfrak{p}}[X])-1.$ It follows from Lemma \ref{Triv}(2) that
 $$\mathrm{Int}(E,D)/\mathfrak{Q}=\mathrm{Int}(E,D)/\left(\mathfrak{p}D_{\mathfrak{p}}[X]\cap \mathrm{Int}(E,D)\right)\cong\mathrm{Int}(E/\mathfrak{p},D/\mathfrak{p}),$$ and so  
$\dim(\mathrm{Int}(E,D)/\mathfrak{Q})=\dim(\mathrm{Int}(E/\mathfrak{p},D/\mathfrak{p})).$
Therefore,
 $$\mathrm{ht}_{D_{\mathfrak{p}}[X]}(\mathfrak{Q})+\dim(\mathrm{Int}(E,D)/\mathfrak{Q})=\dim(D_{\mathfrak{p}}[X])-1+\dim(\mathrm{Int}(E/\mathfrak{p},D/\mathfrak{p})).$$

\textbf{Case 2:} $\mathfrak{Q}$ strictly contains $\mathfrak{p}[X]$. In this case, $\mathfrak{Q}$ is maximal, and then  $\mathrm{ht}_{D_{\mathfrak{p}}[X]}(\mathfrak{Q})\leqslant\dim(D_{\mathfrak{p}}[X])$ and $\dim(\mathrm{Int}(E,D)/\left(\mathfrak{Q}\cap \mathrm{Int}(E,D)\right))\leqslant \dim(\mathrm{Int}(E/\mathfrak{p},D/\mathfrak{p}))-1.$ Hence, $$\mathrm{ht}_{D_{\mathfrak{p}}[X]}(\mathfrak{Q})+\dim(\mathrm{Int}(E,D)/\left(\mathfrak{Q}\cap \mathrm{Int}(E,D)\right))\leqslant \dim(D_{\mathfrak{p}}[X])-1+\dim(\mathrm{Int}(E/\mathfrak{p},D/\mathfrak{p})).$$

Consequently, we obtain the first inequality, $$\dim(\mathrm{Int}(E,D))\leqslant \dim(D_{\mathfrak{p}}[X])-1+\dim(\mathrm{Int}(E/\mathfrak{p},D/\mathfrak{p})).$$

\medskip

For the other inequality, we have
\begin{align*}
\dim(\mathrm{Int}(E,D))&\geqslant \dim(\mathrm{Int}(E,D)/\mathfrak{p}[X])+\mathrm{ht}_{\mathrm{Int}(E,D)}(\mathfrak{p}[X])\\
&=\dim(\mathrm{Int}(E/\mathfrak{p},D/\mathfrak{p}))+\mathrm{ht}_{\mathrm{Int}(E,D)}(\mathfrak{p}[X])
\end{align*}
 
Since $\mathrm{ht}_{\mathrm{Int}(E,D)}(\mathfrak{p}[X])\geqslant \mathrm{ht}_{D_{\mathfrak{p}}[X]}(\mathfrak{p}[X]),$ we deduce that  $$\dim(\mathrm{Int}(E,D))\geqslant \dim(D_{\mathfrak{p}}[X])-1+\dim(\mathrm{Int}(E/\mathfrak{p},D/\mathfrak{p})),$$ and so the equality is proved. For the moreover statement, since $\mathfrak{p}$ is maximal, $D/\mathfrak{p}$ is a field, and then  $\mathrm{Int}(E/\mathfrak{p},D/\mathfrak{p})=(D/\mathfrak{p})[X].$ Thus, $\dim(\mathrm{Int}(E/\mathfrak{p},D/\mathfrak{p}))=\dim((D/\mathfrak{p})[X])=1,$ and therefore $\dim(\mathrm{Int}(E,D))=\dim(D_{\mathfrak{p}}[X])$, which implies that $\dim(\mathrm{Int}(E,D))\leqslant \dim(D[X])$. 
\end{proof}

\medskip

To compute or bound the Krull dimension of $\mathrm{Int}(D)$, or more generally $\mathrm{Int}(E,D)$, it can sometimes be useful to find an extension $B$ of $D$ such that $\mathrm{Int}(E,D) \subseteq B[X]$.

Our next result shows that, under certain conditions, this inclusion guarantees that $\mathrm{Int}(E,D)$ and $B[X]$ have the same Krull dimension. In particular, this extends \cite[Proposition V.4.4]{C97} to the case of $\mathrm{Int}(E,D)$ and simultaneously generalizes \cite [Proposition 1.12]{FIKT97}. Furthermore, we note that the ring $I(B, A)$ studied in \cite{FIKT97} can be viewed as the $\mathrm{Int}(D)$-analogue of the ring $\mathrm{Int}_B(E, D)$ investigated later in this paper, which is closely related to the next theorem.  Before proceeding, we recall a crucial fact from \cite[Section V.4, page 111]{C97} (see also  \cite[pages 337 and 338]{C90}): if $R \subset T$ is a pair of integral domains sharing a nonzero ideal $J$, then there exists a one-to-one correspondence between the prime ideals of $R$ that do not contain $J$ and the prime ideals of $T$ that do not contain $J$. Specifically, each such prime ideal of $T$ corresponds to its contraction in $R$.

\begin{theorem}\label{pairs domain}
Let $D \subset B$ be a pair of integral domains sharing an ideal $I$ such that $D/I$ is finite, and let $E$ be a nonempty subset of $D$. If $\mathrm{Int}(E,D) \subset B[X]$, then $\mathrm{Int}(E,D)$ and $B[X]$ have the same Krull dimension.
\end{theorem}

\begin{proof}
Assume that $\mathrm{Int}(E,D) \subset B[X]$. Then, from the fact that $\mathrm{Int}(E,I)$ is a common ideal of $\mathrm{Int}(E,D)$ and $B[X]$, we can consider the following pullback diagram:
$$\xymatrix{
    \mathrm{Int}(E,D) \ar[r] \ar[d]  & \mathrm{Int}(E,D)/\mathrm{Int}(E,I) \ar[d] \\
    B[X] \ar[r] & B[X]/\mathrm{Int}(E,I).
  }$$    
By \cite[Lemma 1.1]{T97}, we have $\dim( \mathrm{Int}(E,D)/\mathrm{Int}(E,I))=0$, and so it follows from \cite[Corollaire 1 du Théorème 1, page 509]{C88} that $\dim(\mathrm{Int}(E,D))\leqslant\dim(B[X]).$

For the reverse inequality, we will treat the following two possible cases. 

\textbf{Case 1:} $\dim(B[X])=n$. Then, by \cite[Théorème 3, page 35]{J60}, there exists a chain of prime ideals of length $n$ in $B[X]$, ending with $\mathfrak{p}_{n-1}:=\mathfrak{q}[X]\subset \mathfrak{p}_n$, where $\mathfrak{p}_n\cap B=\mathfrak{q}$. Let $\lbrace u_1, \ldots, u_r \rbrace$ denote a complete set of representatives of $D$ modulo $I$. The monic polynomial $f=\prod_{i=1}^{r}(X-u_i)$ belongs to $\int(E,I)$ (in fact, $f$ belongs, a fortiori,  to $\int(D,I)$), and thus, $\mathfrak{q}[X]$ does not contain $\mathrm{Int}(E,I),$ as it does not contain any monic polynomial. Consequently, the considered chain in $B[X]$ contracts to a chain of prime ideals of length $n$ in $\int(E,D)$, and therefore, $n\leqslant\dim(\int(E,D)).$

\textbf{Case 2:} $\dim(B[X])$ is infinite. In this case, $\dim(\mathrm{Int}(E,D))$ is also infinite since a chain of arbitrary length ending at some prime ideal $\mathfrak{p}[X]$ in $B[X]$ contracts to a chain of the same length in $\int(E,D)$. 

Therefore, we conclude that $\dim(\mathrm{Int}(E,D))=\dim(B[X])$.
\end{proof}

We next provide conditions on both $B$ and $E$ that ensure the inclusion $\mathrm{Int}(E,D) \subseteq B[X]$, thereby enabling the application of Theorem \ref{pairs domain}.
\begin{corollary}\label{pairs domainII}
Let $D \subset B$ be a pair of integral domains sharing an ideal $I$ such that $D/I$ is finite and $E$ a nonempty subset of $D$. Assume that there exists a family $\mathcal{P}$ of prime ideals of $B$ that have infinite residue fields and do not contain $I$  such that $B = \cap_{\mathfrak{p}\in\mathcal{P}} B_{\mathfrak{p}}$. If, in addition, that $E$ satisfies one of the following three conditions:
\begin{enumerate}[\hspace*{0.5cm} $(a)$]
\item $E$ meets infinitely many cosets of $\mathfrak{p}$, for each $\mathfrak{p}\in\mathcal{P}.$
\item $E$ is an ideal of $B$ containing $\mathfrak{p}$, for each $\mathfrak{p}\in\mathcal{P}.$
\item $E$ is residually cofinite with $B.$
\end{enumerate}
Then, $\dim(\mathrm{Int}(E,D))=\dim(B[X]).$ 
\end{corollary}

\begin{proof}
For each $\mathfrak{p}\in\mathcal{P}$, we have $D_{\mathfrak{p}}= B_{\mathfrak{p}}$ since the complement of $\mathfrak{p}$ meets the ideal $I$. It follows from Lemma \ref{Triv}(1) that $\mathrm{Int}(E,B)_\mathfrak{p}=B_{\mathfrak{p}}[X],$ and then $\mathrm{Int}(E,D)\subseteq \mathrm{Int}(E,B)_{\mathfrak{p}}=B_{\mathfrak{p}}[X]$  because $\mathrm{Int}(E,D)\subseteq \mathrm{Int}(E,B)$. Thus, $\mathrm{Int}(E,D)\subseteq \cap_{\mathfrak{p}\in \mathcal{P}}B_{\mathfrak{p}}[X]=B[X].$
This inclusion is actually strict. Indeed, if not, the equality $\mathrm{Int}(E,D)=B[X]$ implies that $B$ is an overring of $D,$ and hence, by taking the intersection with the quotient field of $D,$ we obtain $D=B$, which contradicts the assumption that $D\neq B.$ \\ 
Therefore, by Theorem \ref{pairs domain}, we conclude that  $\dim(\mathrm{Int}(E,D)) = \dim(B[X]).$
\end{proof}

To present our second corollary, it is convenient to recall some few concepts. An integral domain $D$ is said to be an \textit{almost Krull} domain if the localization of $D$ at any maximal ideal is a Krull domain. Clearly, Krull domains are almost Krull domains. For an integral domain $D$, the $t$-\textit{operation} is defined by $I_t:=\bigcup (J^{-1})^{-1}$, where $J$ ranges over the set of all nonzero finitely generated ideals contained in $I$. A nonzero ideal $I$ of $D$ is called a $t$-\textit{prime ideal} if it is prime and $I_t=I$.  It is worth noting that any height-one prime is $t$-prime. When each $t$-prime ideal of $D$ has height one, we say that $D$ has $t$-\textit{dimension one} and we denote by $t$-$\dim(D)=1$. Notice that Krull domains and one-dimensional domains have $t$-dimension one. Additionally, it is known that if $D$ an integral domain distinct from its quotient field that is either almost Krull or has $t$-dimension one, then $D=\cap_{\mathfrak{p}\in X^1(D)}D_\mathfrak{p},$ where $X^1(D)$ denotes the set of all height-one prime ideals of $D$.

\begin{corollary}\label{CorBKrull}
Let $D \subset B$ be a pair of integral domains sharing an ideal $I$ such that $D/I$ is finite and $E$ a nonempty subset of $D$. If $B=\cap_{\mathfrak{p}\in X^1(B)}B_\mathfrak{p}$ (this holds, for example, when $B$ is almost Krull or $(t\text{-})\dim(B)=1$) and $\mathrm{ht}_{B}(I)\geqslant2,$ the residue field of each height-one prime ideal of $B$ is infinite and that $E$ verifies one of the  conditions of  Corollary \ref{pairs domainII}. Then, $\dim(\mathrm{Int}(E,D))=\dim(B[X])$.
\end{corollary}
\begin{proof}
This is an application of Corollary \ref{pairs domainII} with $\mathcal{P}=X^1(B).$
\end{proof}

As an explicit illustrative example, we retake  \cite[Example V.4.3]{C97} with two indeterminates, as follows:
\begin{example}
Let $F$ be a field containing a proper finite subfield $k.$ Let us consider the following two domains $B=F[u,v]$ and $D=k+I$, where $u$ and $v$ are two indeterminates over $F,$ and $I=(u,v)$. Let $E$ be a nonempty subset of $D$ that satisfies one of the conditions mentioned in Corollary \ref{pairs domainII}. As noted in \cite[Example V.4.3]{C97}, $D \subset B$ forms a pair of integral domains sharing the ideal $I$. Since $B$ is a Noetherian Krull domain and $I$ has height two in $B$, it follows from Corollary \ref{CorBKrull} that $\dim(\mathrm{Int}(E,D))=\dim(B[X])=1+\dim(B)=1+2=3$.
\end{example}

Recall that an integral domain $D$ is said to be of \textit{finite character}, if every nonzero element of $D$ belongs to only finitely many maximal ideals of $D$. We next provide a generalization of \cite[Proposition 1.7]{T97}.

\begin{theorem}\label{lpsvlf}
Let $D$ be an integral domain with finite residue fields and $E$ a nonempty subset of $D$. Assume that $D$ has finite character and for each maximal ideal $\mathfrak{m}$ of $D$, the integral domain $D_\mathfrak{m}$ shares the maximal ideal $\mathfrak{m}D_\mathfrak{m}$ with an overring $W_\mathfrak{m}$, such that $\mathfrak{m}D_\mathfrak{m}$ is a prime ideal in $W_\mathfrak{m}$. Then, $$ \dim(\mathrm{Int}(E,D))\leqslant \sup\lbrace 1+\dim_v(W_{\mathfrak{m}});\;\mathfrak{m}\in \mathrm{Max}(D) \rbrace.$$
\end{theorem}
 
To prove this result, we first need some preliminary lemmas.

\begin{lemma}[{\cite[Lemma 1.6]{T97}}]\label{tlf}
Under the notation and assumptions of Theorem \ref{lpsvlf} with $D$ is not necessarily assumed to have finite residue fields, we fix a maximal ideal $\mathfrak{m}_{0}$ of $D$ and we set $C=C(\mathfrak{m}_{0})=W_{\mathfrak{m}_{0}}\cap (\cap_{\mathfrak{m}\neq \mathfrak{m}_{0}}D_\mathfrak{m}).$ Then, $\mathfrak{m}_{0}$ is a prime ideal in $C$ and $C_{\mathfrak{m}_{0}}=W_{\mathfrak{m}_{0}}.$
\end{lemma}

The following result follows from \cite[Proposition 1.2]{T97} and its proof.
\begin{lemma}\label{lemIneq}
Let $D\subset B$ be integral domains  sharing a common nonzero ideal $I$ and $E$ a nonempty subset of $K.$ If $D/I$ is finite, then $$\dim(\mathrm{Int}(E,D))\leqslant \dim(\mathrm{Int}(E,B)).$$
\end{lemma}

\begin{lemma}\label{LemLocalDim}
For any subset $E$ of $D$, we have $$\dim(\mathrm{Int}(E,D))=\sup\left\{\dim(\mathrm{Int}(E,D)_\mathfrak{m});\,\mathfrak{m}\in\mathrm{Max}(D)\right\}.$$ 
\end{lemma}
\begin{proof}
This is a particular case of \cite[Lemma 1.2]{Tam21}.
\end{proof}

\begin{proof}[Proof of Theorem \ref{lpsvlf}]
Let $\mathfrak{m}$ be a maximal ideal of $D$. It follows from  Lemma \ref{tlf} that $C(\mathfrak{m})$ and $D$ share the ideal $\mathfrak{m}$, and then $\mathrm{Int}(E,D)$ and $\mathrm{Int}(E,C(\mathfrak{m}))$ share the ideal $\mathrm{Int}(E,\mathfrak{m})$.  Thus, $\mathrm{Int}(E,D)_{\mathfrak{m}}$ and $\mathrm{Int}(E,C(\mathfrak{m}))_{\mathfrak{m}}$ share the ideal $\mathrm{Int}(E,\mathfrak{m})_{\mathfrak{m}}$. Since $D/\mathfrak{m}$ is finite, we have  $\dim(\mathrm{Int}(E,D)_{\mathfrak{m}})\leqslant \dim(\mathrm{Int}(E,C(\mathfrak{m}))_{\mathfrak{m}})$ as asserted Lemma \ref{lemIneq}. From the inclusion $C(\mathfrak{m})[X]\subseteq \mathrm{Int}(E,C(\mathfrak{m}))$ and by using  Lemma \ref{tlf}, we deduce that $(C(\mathfrak{m}))_{\mathfrak{m}}[X]=W_{\mathfrak{m}}[X]\subseteq \mathrm{Int}(E,C(\mathfrak{m}))_{\mathfrak{m}},$ and hence, $\dim(\mathrm{Int}(E,D)_{\mathfrak{m}})\leqslant \dim(\mathrm{Int}(E,C(\mathfrak{m}))_{\mathfrak{m}})\leqslant \dim_{v}(W_{\mathfrak{m}}[X])=1+\dim_{v}(W_{\mathfrak{m}}).$ Thus, the thesis follows from Lemma \ref{LemLocalDim}.
\end{proof}

Following \cite{HH78}, an integral domain $D$ is called a \textit{pseudo-valuation domain} (in short, a PVD) if every prime ideal $\mathfrak{p}$ has the property that whenever a product of two elements of the quotient field of $D$ lies in $\mathfrak{p},$ then one of the given elements is in $\mathfrak{p}$. In \cite{HH78}, the authors showed that valuation domains form a proper subclass of PVDs. Additionally, they established that if a PVD $D$ that is not a valuation domain $D$ must be a local domain whose maximal ideal $\mathfrak{m}$ is shared with a unique valuation overring $V$, known as the \textit{associated valuation domain} of $D$.  Lastly, recall that an integral domain $D$ is said to be \textit{locally PVD} if $D_\mathfrak{m}$ is a PVD for each maximal ideal $\mathfrak{m}$ of $D$. Note that while PVDs and Pr\"ufer domains are locally PVD, a locally PVD  need not be a PVD or a Pr\"ufer domain, as illustrated in \cite[Example 2.5]{DF83}.

\begin{corollary}\label{LPVDs}
Let $D$ be a locally PVD of finite character and $E$ is a  nonempty subset of $D$. If $D$ has finite residue fields, then $$\dim(\mathrm{Int}(E,D))=1+\dim(D).$$ 
\end{corollary}

\begin{proof}
In this situation, for each maximal ideal $\mathfrak{m}$ of $D_{\mathfrak{m}}$, we have $D_{\mathfrak{m}}$ is a PVD with finite residue field, and let $W_{\mathfrak{m}}$ denote its associated valuation domain. Then,  $\dim_v(W_{\mathfrak{m}}) = \dim(W_{\mathfrak{m}}) = \dim(D_{\mathfrak{m}}),$ and hence it follows from Theorem \ref{lpsvlf} that
$$\dim(\mathrm{Int}(E,D)) \leqslant \sup\left\{ 1 + \dim(D_{\mathfrak{m}}); \, \mathfrak{m} \in \mathrm{Max}(D) \right\} = 1 + \dim(D).$$
For the other inequality, it is well-known for any fractional subset $E$ of $D$; see \cite[Proposition V.1.5]{C97}, and so the desired equality follows.
\end{proof}

The following result extends Theorem \ref{lpsvlf} by removing the hypothesis that $D$ must have finite residue fields. Let us first introduce the notation: $\mathcal{M}_0$  denotes the set of all maximal ideals of $D$ with finite residue fields, while $\mathcal{M}_1$ denotes the set of those with infinite residue fields.

\begin{proposition}\label{PropM}
Under the notation and assumptions of Theorem \ref{lpsvlf} with $D$ is not necessarily assumed to have finite residue fields. Assume that, for each maximal ideal $\mathfrak{m}$ of $D$ with infinite residue field, the subset $E$ satisfies one of the following three conditions:
\begin{enumerate}[\hspace*{0.5cm} $(a)$]
\item $E$ meets infinitely many cosets of $\mathfrak{m}$.
\item $E$ is an ideal of $D$ contains $\mathfrak{m}$.
\item $E$ is residually cofinite with $D$.
\end{enumerate}
Then, we have : $$\dim(\mathrm{Int}(E,D))\leqslant \sup \lbrace\sup\lbrace 1+\dim_v(W_{\mathfrak{m}});\;\mathfrak{m}\in\mathcal{M}_0\rbrace , \sup\lbrace \dim(D_{\mathfrak{m}}[X]);\;\mathfrak{m}\in\mathcal{M}_1\rbrace \rbrace.$$ 
\end{proposition}
\begin{proof}
By Lemma \ref{LemLocalDim}, $\dim(\mathrm{Int}(E,D))=\sup\left\{\dim(\mathrm{Int}(E,D)_\mathfrak{m});\,\mathfrak{m}\in\mathrm{Max}(D)\right\}.$ So, let $\mathfrak{m}$ be a maximal ideal of $D$. If $D/\mathfrak{m}$ is infinite, then, by Lemma \ref{Triv}, we have $\mathrm{Int}(E,D)_\mathfrak{m}=D_{\mathfrak{m}}[X]$. For the case where $D/\mathfrak{m}$ is finite, using the same arguments as in the proof of Theorem \ref{lpsvlf}, we deduce that $\dim(\mathrm{Int}(E,D)_{\mathfrak{m}})\leqslant 1+\dim_{v}(W_{\mathfrak{m}}),$ which completes the proof.
\end{proof}

From the previous result, we obtain the following corollary, which generalizes \cite[Proposition 1.8]{T97}.

\begin{corollary}\label{psvf}
Let $D$ be a locally PVD of finite character and $E$ a  nonempty subset of $D$ that verifies one of the conditions cited for $E$ in Proposition \ref{PropM}. Then, $\dim(\mathrm{Int}(E,D))=\lambda +\dim(D),$ where $\lambda$ equals $1$ or $2$. 
\end{corollary}

\begin{proof}
As in the proof of Corollary \ref{LPVDs}, we have  $\dim_v(W_{\mathfrak{m}})=\dim(W_{\mathfrak{m}})=\dim(D_{\mathfrak{m}}),$  for each maximal ideal $\mathfrak{m}$ of $D.$ So, let $\mathfrak{m}$ be a maximal ideal of $D,$ and then two cases are possible:

If $\mathfrak{m}\in\mathcal{M}_0,$ then $\dim(\mathrm{Int}(E,D)_{\mathfrak{m}})\leqslant 1+\dim_{v}(W_{\mathfrak{m}})=1+\dim(W_{\mathfrak{m}}),$ and so  $\dim(\mathrm{Int}(E,D)_{\mathfrak{m}})\leqslant 1+\dim(D_{\mathfrak{m}}).$

If $\mathfrak{m}\in\mathcal{M}_1,$ then $\dim(\mathrm{Int}(E,D)_{\mathfrak{m}})=\dim(D_{\mathfrak{m}}[X]).$ Hence, from \cite[Theorem 2.5]{HH2} and \cite[Theorem 39]{K74}, we deduce that $\dim(\mathrm{Int}(E,D)_{\mathfrak{m}})$ is either $1 + \dim(D_{\mathfrak{m}})$ or $2 + \dim(D_{\mathfrak{m}})$.

Consequently, in all cases, we have $\dim(\mathrm{Int}(E,D)_{\mathfrak{m}}) \leqslant 2 + \dim(D_{\mathfrak{m}})$. Therefore, by Proposition \ref{PropM} and \cite[Proposition V.1.5]{C97}, it follows that  $$1 + \dim(D) \leqslant \dim(\mathrm{Int}(E,D)) \leqslant 2 + \dim(D),$$ which completes the proof.
\end{proof}

The aim of the remaining part of this paper is to establish bounds on the Krull dimension of $\mathrm{Int}_B(E,D)$. Specifically, we will present our last main theorem, which extends Lemma \ref{lemIneq}, and subsequently illustrate its application through two examples. Before proceeding, we introduce a crucial lemma analogous to \cite[Lemma 1.1]{T97} for $\mathrm{Int}_B(E,D)$, whose proof is inspired by the aforementioned lemma.

\begin{lemma}\label{cidpol}
Let $I$ be an ideal of $D,$ $B$ an overring of $D$ and $E$ a nonempty subset of $K$. Set $\mathrm{Int}_B(E,I):=\{f\in B[X];\;f(E)\subseteq I\}.$ We have 
\begin{enumerate}[$(1)$]
\item $\mathrm{Int}_B(E,I)$ is an ideal of $\mathrm{Int}_B(E,D)$. 
\item If $D/I$ is finite, then prime ideals of $\mathrm{Int}_B(E,D)$ containing $\mathrm{Int}_B(E,I)$ are maximal, and hence the ring $\mathrm{Int}_B(E,D)/\mathrm{Int}_B(E,I)$ is zero dimensional. 
\end{enumerate}
\end{lemma}

\begin{proof}
Statement (1) is clear. For statement $(2)$, assume that $D/I$ is finite with cardinality $q$  and let $\mathfrak{P}$ be a prime ideal of $\mathrm{Int}_B(E,D)$ containing $\mathrm{Int}_B(E,I)$. We aim to prove that $\mathfrak{P}$ is maximal.  To achieve this, consider a set of representatives $\{ u_1, \dots, u_q \}$ of $D$ modulo $I$ (where $q$ is the quardinality of $D/I$) and arguing similarly to the proof of \cite[Lemma V.1.9]{C97}. So, let $f$ be an element of $\int_B(E, D)$. Then, the product $\prod_{i=1}^{r}(f(X) - u_i)$ lies in $\int_B(E, I)$, and hence, in $\mathfrak{P}$. Therefore, one of its factors must also lie in $\mathfrak{P}$, which means $\{ u_1, \dots, u_r \}$ is a set of representatives of $\int_B(E, D)$ modulo $\mathfrak{P}$. Consequently, $\int_B(E, D)/\mathfrak{P}$ is a field, which is equivalent to stating that $\mathfrak{P}$ is maximal. Hence, we deduce that $\dim(\int_B(E, D)/\int_B(E, I)) = 0.$
\end{proof}

\begin{theorem}\label{main1}
Let $D\subset R\subset B$ be integral domains such that $D\subset R$ share a common nonzero ideal $I$ and $E$ a nonempty subset of the quotient field of $D.$ If $D/I$ is finite, then $$\dim(\mathrm{Int}_B(E,D))\leqslant \dim(\mathrm{Int}_B(E,R)).$$ If, in addition, $E$ is a fractional subset of $D$ contained in $R$, then $$1+\dim(D)\leqslant\dim(\mathrm{Int}_B(E,D))\leqslant 1+\dim_v(R).$$
\end{theorem}

\begin{proof}
It is clear that $\mathrm{Int}_B(E,D)\subset\mathrm{Int}_B(E,R))$ and $\mathrm{Int}_B(E,I)$ is a common ideal of $\mathrm{Int}_B(E,D)$ and $\mathrm{Int}_B(E,R),$ so we can consider the following pullback diagram:
$$\xymatrix{
    \mathrm{Int}_B(E,D) \ar[r] \ar[d]  & \mathrm{Int}_B(E,D)/\mathrm{Int}_B(E,I) \ar[d] \\
    \mathrm{Int}_B(E,R) \ar[r] & \mathrm{Int}_B(E,R)/\mathrm{Int}_B(E,I).
  }$$ 
  Thus, combining \cite[Corollaire 1 du Théorème 1, page 509]{C88} with Lemma \ref{cidpol}, we deduce that
$$\dim(\mathrm{Int}_B(E,D))\leqslant\dim(\mathrm{Int}_B(E,R)).$$ 
The additional statement follows from \cite[Corollary 2.8]{CT23} for the left inequality  and from the inclusion $R[X]\subseteq \mathrm{Int}_B(E,R)$ for the right inequality. 
\end{proof}

A finite Krull dimensional domain $D$ is said to be \textit{Jaffard} if $\dim_v(D) = \dim(D).$ For instance, in the finite Krull dimensional case, Noetherian domains and Pr\"ufer domains are examples of Jaffard domains.

\begin{corollary}\label{c11}
Under the notation and assumptions of Theorem \ref{main1} with $R$ is assumed to be a Jaffard domain, the following statements hold.
\begin{enumerate}[$(1)$]
\item $\dim(D)\leqslant \dim(\mathrm{Int}_B(E,D))\leqslant 1+\dim(R).$
\item If $E$ is a fractional subset of $D$ contained in $R$, then $$1+\dim(D)\leqslant \dim(\mathrm{Int}_B(E,D))\leqslant 1+\dim(R).$$
\end{enumerate}
\end{corollary}
\begin{proof}
(1) The left inequality is established in \cite[Corollary 2.8]{CT23}. For the right inequality, since $R$ is a Jaffard domain, we have $\dim(\mathrm{Int}_B(E,R))= 1+\dim(R)$, and thus the desired inequality follows from Theorem \ref{main1}.  \\
(2) This follows by combining \cite[Corollary 2.8]{CT23} with the previous statement.
\end{proof}

As a consequence of Corollary \ref{c11}, we derive the following result:
\begin{corollary}\label{CorPVD}
Let $D$ be a PVD with a finite residue field and $V$ its associated valuation domain. For any overring $B$ of $V$ and any nonempty subset $E$ of the quotient field of $D$, we have:
\begin{enumerate}[$(1)$]
\item $\dim(\mathrm{Int}_B(E,D))=\lambda +\dim(D),$ where $\lambda$ equals $0$ or $1$.
\item If $E$ is a fractional subset of $D$ contained in $V$, then $$ \dim(\mathrm{Int}_B(E,D))= 1+\dim(D).$$
\end{enumerate}
\end{corollary}

\begin{proof}
In this case, $R$ corresponds to the valuation domain $V$, and so we are in the hypotheses of Corollary \ref{c11}. Consequently, the desired conclusion follows from the equality $\dim(D)=\dim(V).$
\end{proof}

We conclude this paper by providing two examples that effectively illustrate Corollary \ref{CorPVD}, including cases of integral domains $D$ that are PVDs but not Jaffard domains.

\begin{example}
Let $k$ be a finite field and $L$ a field containing $k$ such that the extension $L/k$ is transcendental. Consider the integral domain  $D:=k+TL[[T]],$ where $T$ is an indeterminate over $L$. It is known to be a one-dimensional PVD with a finite residue field (which is not a Jaffard domain by \cite[Proposition 2.5\rm(b)]{ABDFK88}). Thus, by applying Corollary \ref{CorPVD}, the Krull dimension of $\mathrm{Int}_{B}(E,D)$ is equal to two, for any  nonempty subset $E$ of $D$ and any overring $B$ of $V:=L[[T]].$ 
\end{example}

\begin{example}
 Let  $Y$ and $Z$ be two indeterminates over a finite field $k$. Set $D:=k+Zk(Y)[Z]_{(Z)}$ and $V:=k(Y)[Z]_{(Z)},$ and let $E$ a fractional subset of $D$ and $B$ an overring of $V$. It is known that $D$ is a one-dimensional PVD that is not Jaffard because $k[Y]_{(Y)}+Zk(Y)[Z]_{(Z)}$ is a two-dimensional valuation overring of $D.$ Hence, by Corollary \ref{CorPVD},   $\dim(\int_B(E,D))=1+\dim(D)=2.$ Moreover, it follows from \cite[Proposition 2.9(2)]{CT23} that $\dim_v(\int_B(E,D))=1+\dim_v(D)=3,$ and thus $\int_B(E,D)$ is  not Jaffard by \cite[Proposition 2.9(3)]{CT23}. 
\end{example}

\textbf{Acknowledgements}. The authors sincerely thank the referee for the helpful comments and suggestions that contributed to improving the exposition of this paper.

\end{document}